\documentclass[a4paper,english]{extarticle}
\usepackage[utf8]{luainputenc}
\usepackage{babel}
\usepackage{amsthm}
\usepackage{amsmath}
\usepackage{amssymb}
\usepackage{esint}
\usepackage[unicode=true,pdfusetitle,
 bookmarks=true,bookmarksnumbered=false,bookmarksopen=false,
 breaklinks=false,pdfborder={0 0 1},backref=false,colorlinks=false]
 {hyperref}

\makeatletter

\pdfpageheight\paperheight
\pdfpagewidth\paperwidth

\numberwithin{equation}{section}
\numberwithin{figure}{section}
\theoremstyle{plain}
\newtheorem{thm}{\protect\theoremname}[section]
\theoremstyle{definition}
\newtheorem{defn}[thm]{\protect\definitionname}
\theoremstyle{plain}
\newtheorem{lem}[thm]{\protect\lemmaname}
\ifx\proof\undefined
\newenvironment{proof}[1][\protect\proofname]{\par
\normalfont\topsep6\p@\@plus6\p@\relax
\trivlist
\itemindent\parindent
\item[\hskip\labelsep\scshape #1]\ignorespaces
}{%
\endtrivlist\@endpefalse
}
\providecommand{\proofname}{Proof}
\fi
\theoremstyle{plain}
\newtheorem{cor}[thm]{\protect\corollaryname}
\theoremstyle{remark}
\newtheorem{rem}[thm]{\protect\remarkname}
\theoremstyle{plain}
\newtheorem{prop}[thm]{\protect\propositionname}

\usepackage{cancel}
\numberwithin{equation}{section}
\date{}

\makeatother

\providecommand{\corollaryname}{Corollary}
\providecommand{\definitionname}{Definition}
\providecommand{\lemmaname}{Lemma}
\providecommand{\propositionname}{Proposition}
\providecommand{\remarkname}{Remark}
\providecommand{\theoremname}{Theorem}

\begin{document}

\title{Strong Unique Continuation for a Residual Stress System with Gevrey
Coefficients}

\author{Yi-Hsuan Lin}
\maketitle
\begin{abstract}
We consider the problem of the strong unique continuation for an elasticity
system with general residual stress. Due to the known counterexamples,
we assume the coefficients of the elasticity system are in the Gevrey
class of appropriate indices. The main tools are Carleman estimates
for product of two second order elliptic operators.
\end{abstract}

\section{Introduction and statement of the results}

In this paper, we prove the strong unique continuation property (SUCP)
for the isotropic elasticity system with residual stress under appropriate
conditions. We formulate the mathematical problem in the following.

Let $\Omega$ be a connected open domain in $\mathbb{R}^{3}$ and
consider the time-harmonic elasticity system 
\begin{equation}
\nabla\cdot\sigma+\kappa^{2}\rho u=0\mbox{ in }\Omega,\label{eq:1.1}
\end{equation}
where $\sigma=(\sigma_{ij})_{i,j=1}^{3}$ is the stress tensor field,
$\kappa\in\mathbb{R}$ is the frequency and $\rho=\rho(x)>0$ denotes
the density of the medium. The vector field $u(x)=(u_{i}(x))_{i=1}^{3}$
is the displacement vector. Suppose that the stress tensor is given
by 
\[
\sigma(x)=T(x)+(\nabla u)T(x)+\lambda(x)(\mathrm{tr}E)I+2\mu(x)E,
\]
where $E(x)=\dfrac{\nabla u+\nabla u^{t}}{2}$ is the infinitesimal
strain and $\lambda(x),\mu(x)$ are the Lam$\acute{\mathrm{e}}$ parameters.
The second-rank tensor $T(x)=(t_{ij}(x))_{i,j=1}^{3}$ is the residual
stress and satisfies 
\[
t_{ij}(x)=t_{ji}(x),\mbox{ }\forall i,j=1,2,3\mbox{ and }x\in\Omega
\]
 and 
\[
\nabla\cdot T=\sum_{j}\partial_{j}t_{ij}=0\mbox{ in }\Omega,\mbox{ }\forall i=1,2,3.
\]
If we define the elastic tensor $C=(C_{ijkl})_{i,j,k,l=1}^{3}$ with
\[
C_{ijkl}=\lambda\delta_{ij}\delta_{kl}+\mu(\delta_{jk}\delta_{jl}+\delta_{jk}\delta_{il})+t_{jl}\delta_{ik},
\]
then (\ref{eq:1.1}) is equivalent to 
\[
\nabla\cdot(C\nabla u)+\kappa^{2}\rho u=0\mbox{ in }\Omega.
\]

We concern the SUCP for (\ref{eq:1.1}), i.e., if $u\in H_{loc}^{2}(\Omega)$
satisfies (\ref{eq:1.1}) and $u(x)$ vanishes to infinite order at
a point $x_{0}\in\Omega$, then $u$ must vanish identically in $\Omega$.
Without loss of generality, we assume $x_{0}=0$. A brief history
of the results on the (strong) unique continuation for (\ref{eq:1.1})
is in the following. In \cite{nakamura2003unique}, Nakamura and Wang
proved the unique continuation property for (\ref{eq:1.1}) under
the condition $\max_{i,j}\|t_{ij}\|_{\infty}$ is small and $T(x),\lambda(x)$,
$\mu(x)\in W^{2,\infty}$ and $\rho(x)\in W^{1,\infty}$. In \cite{lin2004strong},
Lin proved the SUCP for (\ref{eq:1.1}) under the assumptions that
$T(0)=0$, $\max_{i,j}\|t_{ij}\|_{\infty}$ is small, $\lambda(x),\mu(x)$
and $\rho(x)$ are in $C^{2}$. In addition, in \cite{uhlmann2009unique},
Uhlmann and Wang proved unique continuation principle for (\ref{eq:1.1})
under the conditions $T(x),\lambda(x)$, $\mu(x)\in W^{2,\infty}$,
$\rho(x)\in W^{1,\infty}$ and general residual stress.

Motivated by \cite{uhlmann2009unique}, we want to prove the SUCP
for (\ref{eq:1.1}) with arbitrary residual stress. In this paper,
we will give a reduction algorithm to transform (\ref{eq:1.1}) into
a special fourth order elliptic system. The main difficulty is that
when $T(0)\neq0$, the leading terms of (\ref{eq:1.1}) will not be
the Laplacian at zero, so we cannot use a perturbation argument to
derive suitable Carleman estimates in order to obtain the SUCP. In
\cite{alinhac1980uniqueness}, Alinhac and Baouendi proved the SUCP
for any fourth order operator with smooth coefficients verifying $P=Q_{2}Q_{1}+R$,
where $Q_{i}$'s are second order elliptic operators with $Q_{i}(0,D)=-\Delta$
for $i=1,2$. Moreover, in \cite{le2001strong}, Le Borgne proved
the SUCP for fourth order differential inequality with $Q_{i}$'s
are Lipschitz continuous and $Q_{i}(0,D)=-\Delta$ for $i=1,2$. In
\cite{lin2004strong}, Lin introduced $v=\nabla\cdot u$ and $w=\nabla\times u$
to transform (\ref{eq:1.1}) into a second order differential system,
but the system is weakly-coupled, i.e., the principal part of the
second order derivatives are not diagonal. Moreover, Lin also introduced
a fourth order elliptic system $P=\Delta Q_{i}$ with $Q_{i}$'s are
second order elliptic operators with $Q_{i}(0,D)=\Delta$ for $i=1,2$
and give another approach to derive the SUCP. For more details, we
refer readers to \cite{lin2004strong}.

In this note, our transformation will reduce (\ref{eq:1.1}) into
a fourth order principally diagonal elliptic system with the same
leading coefficients. The key observation is that the leading terms
of the fourth order elliptic system are the same. Notice that principally
diagonal strongly elliptic systems allow the application of Carleman
estimates for scalar operators since these estimates are flexible
with respect to perturbations by lower order terms. Therefore, it
is possible to derive suitable Carleman estimates for the fourth order
elliptic system.

In general, the SUCP doe not hold even the coefficients are smooth,
Alinhac gave a counterexample in \cite{alinhac1980non}. Thus, we
consider all the coefficients in the Gevrey class and we will use
the Carleman estimates proved in \cite{colombini2010strong} for the
scalar higher order elliptic equations in order to prove the SUCP
for the new fourth order strongly elliptic system.
\begin{defn}
We say that $f\in C^{\infty}(\Omega)$ belongs to the Gevrey class
of order $s$, denote it as $G^{s}(\Omega)$, if there exist constants
$c$, $A$ and multiindices $\beta$ such that
\[
|\partial^{\beta}f|\leq cA^{|\beta|}|\beta|!^{s}\mbox{ in }\Omega.
\]
To simplify the notation, from now on, we use $G^{s}$ to denote $G^{s}(\Omega)$.
In this paper, we assume all the coefficients $T(x),\lambda(x),\mu(x)$
and $\rho(x)$ lie in the Gevrey class $G^{s}$. We are interested
in the SUCP for (\ref{eq:1.1}) with Gevery coefficients, which means
if $u$ satisfies (\ref{eq:1.1}) and $u$ is flat at the origin in
the sense that 
\begin{equation}
\sup_{r\leq\delta}r^{-N}\|u\|_{L^{2}(B(0,r))}<\infty\label{eq:1.2}
\end{equation}
for all $N$, then $u$ vanishes near the origin. If $u$ is smooth,
the condition (\ref{eq:1.2}) is equivalent to all partial derivatives
of $u$ vanishing at $0$.
\end{defn}
The SUCP for the second order elliptic equations in the Gevrey class
were studied in many literature \cite{colombini2005result,colombini2006strong,colombini2010strong,lerner1981resultats}.
In 1981, Lerner \cite{lerner1981resultats} considered a second order
elliptic operator $L$ in $\mathbb{R}^{2}$ with simple characteristics
and the coefficients in the Gevery class of order $s$. Lerner proved
that if $s$ is smaller than a quantity depending on the principal
symbol of $l_{2}(0,\mathbb{R}^{2})$, then $L$ has the SUCP near
$0$. In \cite{colombini2006strong}, the authors extended Lerner's
result to $\mathbb{R}^{N}$, which means the SUCP holds for a second
order elliptic operator $L$ in $\mathbb{R}^{N}$ with the Gevrey
order $s$ smaller than a quantity depending on the principal symbol
of $l_{2}(0,\mathbb{R}^{N})$.

Recall that the strongly elliptic condition is given as: there exists
$c_{0}>0$ such that for all vectors $a=(a_{i})_{i=1}^{3}$, $b=(b_{i})_{i=1}^{3}$
\[
\sum_{ijkl}C_{ijkl}a_{i}b_{j}a_{k}b_{l}\geq c_{0}|a|^{2}|b|^{2}\mbox{ }\forall x\in\Omega.
\]
In this paper, we assume $P_{1}$ and $P_{2}$ are two strongly elliptic
operators, where
\begin{eqnarray}
P_{1}(x,D) & := & \sum_{jk}a_{jk}^{1}(x)\partial_{x_{j}x_{k}}^{2}:=\sum_{jk}(\mu\delta_{jk}+t_{jk})\partial_{x_{j}x_{k}}^{2},\label{eq:1.3}\\
P_{2}(x,D) & := & \sum_{jk}a_{jk}^{2}(x)\partial_{x_{j}x_{k}}^{2}:=\sum_{jk}((\lambda+2\mu)\delta_{jk}+t_{jk})\partial_{x_{j}x_{k}}^{2}\label{eq:1.4}
\end{eqnarray}
 with $a_{jk}^{1}(x)=\mu(x)\delta_{jk}+t_{jk}(x)$ and $a_{jk}^{2}(x)=(\lambda(x)+2\mu(x))\delta_{jk}+t_{jk}(x)$.
Further, there exists $c_{0}>0$ such that for any $\xi=(\xi_{i})_{i=1}^{3}\in\mathbb{R}^{3}$
\begin{eqnarray}
\sum_{jk}a_{jk}^{1}(x)\xi_{j}\xi_{k} & = & \sum_{jk}t_{jk}\xi_{j}\xi_{k}+\mu|\xi|^{2}\geq c_{0}|\xi|^{2}\label{eq:1.5}\\
\sum_{jk}a_{jk}^{2}(x)\xi_{j}\xi_{k} & = & \sum_{jk}t_{jk}\xi_{j}\xi_{k}+(\lambda+2\mu)|\xi|^{2}\geq c_{0}|\xi|^{2}\label{eq:1.6}
\end{eqnarray}
for all $x\in\Omega$, note that $(a_{jk}^{\ell}(x))_{j,k=1}^{3}$
is a symmetric matrix for $\ell=1,2$. 

We also assume that there exists a constant $\alpha>0$ such that
the eigenvalues $\lambda_{1}^{\ell}\leq\lambda_{2}^{\ell}\leq\lambda_{3}^{\ell}$
to be eigenvalues of $(a_{jk}^{\ell}(0))$ satisfying
\begin{equation}
\alpha>\dfrac{\lambda_{3}^{\ell}-\lambda_{1}^{\ell}}{\lambda_{1}^{\ell}}\label{eq:1.7}
\end{equation}
and 
\begin{equation}
s<1+\dfrac{1}{\alpha}\label{eq:1.8}
\end{equation}
uniformly in $x$ and for $\ell=1,2$. 

The following theorem derives the SUCP for (\ref{eq:1.1}) when all
the coefficients lie in the Gevrey class $G^{s}$.
\begin{thm}
Let the residual stress $(t_{ij}(x))_{i,j=1}^{3}$, the Lam$\acute{e}$
parameters $\lambda(x)$, $\mu(x)$ and the density of the medium
$\rho(x)$ be in the Gevrey class $G^{s}(\Omega)$ with $s$ satisfying
(\ref{eq:1.8}). Then for all $u\in H_{loc}^{2}(\Omega;\mathbb{R}^{3})$
solving (\ref{eq:1.1}) and for all $N>0$
\[
\int_{R\leq|x|\leq2R}|u|^{2}dx=O(R^{N})\mbox{ as }R\to0,
\]
then $u$ is identically zero in $\Omega$.
\end{thm}
This paper is organized as follows. In section 2, we will reduce (\ref{eq:1.1})
into a fourth order principally diagonal elliptic system. We use the
ideas in \cite{lin2004strong} and give more detailed transformations.
In section 3, we will use the property of the strongly elliptic system
in the Gevrey class, then we can get the asymptotic behavior of $u$
near 0. In section 4, we state the SUCP for the fourth order elliptic
system and prove the theorem by using the Carleman estimates.

\section{Reduction to a fourth order strongly elliptic system}

In this section, we want to transform (\ref{eq:1.1}) into a principally
diagonal fourth order strongly elliptic system. As the calculation
in \cite{lin2004strong}. Let 
\begin{equation}
Ru=\nabla\cdot(\nabla uT)\label{eq:2.1}
\end{equation}
with $Ru=((Ru)_{1},(Ru)_{2},(Ru)_{3})$, where $(Ru)_{i}=\sum_{jk}t_{jk}\partial_{jk}^{2}u_{i}$,
$i=1,2,3$.

As in Section 2, we set $U=(u,v,w)^{t}$, where $v=\nabla\cdot u$,
$w=\nabla\times u$ and $u$ satisfies (\ref{eq:1.1}). From (\ref{eq:1.1}),
(\ref{eq:2.1}), let $P_{1}$ and $P_{2}$ be two elliptic operators
\begin{eqnarray*}
P_{1}(x,D) & = & R+\mu\Delta,\\
P_{2}(x,D) & = & R+(\lambda+2\mu)\Delta,
\end{eqnarray*}
 then $(u,v,w)$ satisfies
\begin{eqnarray}
P_{1}(x,D)u & = & A_{1,1}(u,v)+A_{1,0}(u,v),\label{eq:2.2.}\\
P_{2}(x,D)v & = & -\sum_{jk}\nabla(t_{jk})\cdot\partial_{jk}^{2}u\label{eq:2.3}\\
 &  & +A_{2,1}(u,v,w)+A_{2,0}(u,v,w),\nonumber \\
P_{1}(x,D)w & = & -\sum_{jk}\nabla(t_{jk})\times\partial_{jk}^{2}u\label{eq:2.4}\\
 &  & +A_{3,1}(u,v,w)+A_{3,2}(u,v,w),\nonumber 
\end{eqnarray}
 where $A_{\ell,m}$ are $m$-th order differential operators. For
more details, we refer reader to \cite{lin2004strong}. 

Notice that $u\in H_{loc}^{2}(\Omega;\mathbb{R}^{3})$ satisfies (\ref{eq:2.2.})
and $v=\nabla\cdot u\in H_{loc}^{1}(\Omega)$ and $\nabla v\in L_{loc}^{2}(\Omega)$,
then the right hand side of (\ref{eq:2.2.}) lies in $L_{loc}^{2}(\Omega)$.
Therefore, we use the standard elliptic higher order regularity theory
for (\ref{eq:2.2.}) (see Theorem 2.2 in \cite{giaquinta1993introduction})
and the strongly elliptic property, then we have $u\in H_{loc}^{3}(\Omega;\mathbb{R}^{3})$.
Iterate the procedures, we obtain $u\in H_{loc}^{k}(\Omega;\mathbb{R}^{3})$
$\forall k\in\mathbb{N}$ (which implies $v,w\in H_{loc}^{k}(\Omega)$
$\forall k\in\mathbb{N}$).

Let $P(x,D)$ be the principal part of the system to get 
\[
P(x,D)U=(P_{1}(x,D)u,P_{2}(x,D)v,P_{1}(x,D)w)^{t},
\]
where $U:=(u,v,w)^{t}:\Omega\to\mathbb{R}^{7}$. Component-wise, we
have
\begin{eqnarray*}
(P(x,D)U)_{i} & = & \mu\Delta u_{i}+\sum_{jk}t_{jk}\partial_{jk}^{2}u_{i},\mbox{ }i=1,2,3\\
(P(x,D)U)_{i} & = & (\lambda+2\mu)\Delta v+\sum_{jk}t_{jk}\partial_{jk}^{2}v,\mbox{ }i=4\\
(P(x,D)U)_{i} & = & \mu\Delta w_{i-4}+\sum_{jk}t_{jk}\partial_{jk}^{2}w_{i-4}\mbox{ }i=5,6,7.
\end{eqnarray*}
Now, let us take the second order elliptic operator $P_{2}(x,D)$
on (\ref{eq:2.2.}), we get 
\begin{eqnarray}
P_{2}P_{1}(x,D)u & = & P_{2}(x,D)[A_{1,1}(u,v)+A_{1,0}(u,v)]\label{eq:2.5}\\
 & := & \sum_{m=0}^{3}B_{1,m}(u,v),\nonumber 
\end{eqnarray}
where $B_{1,m}$ is an $m$-th order differential operator. Similarly,
we can take $P_{1}(x,D)$ on (\ref{eq:2.3}) and $P_{2}(x,D)$ on
(\ref{eq:2.4}), then we obtain 
\begin{align}
 & P_{1}P_{2}(x,D)v\label{eq:2.6}\\
 & =P_{1}(x,D)(-\sum_{jk}\nabla(t_{jk})\cdot\partial_{jk}^{2}u)+P_{1}(x,D)(A_{2,1}(u,v,w)+A_{2,0}(u,v,w))\nonumber \\
 & =-P_{1}(x,D)(\sum_{jk}\nabla(t_{jk})\cdot\partial_{jk}^{2}u)+\sum_{m=0}^{3}B_{2,m}(u,v,w),\nonumber 
\end{align}
and 
\begin{align}
 & P_{2}P_{1}(x,D)w\label{eq:2.7}\\
 & =P_{2}(x,D)(-\sum_{jk}\nabla(t_{jk})\times\partial_{jk}^{2}u)+P_{2}(x,D)(A(u,v,w)+A_{3,2}(u,v,w))\nonumber \\
 & =-P_{2}(x,D)(\sum_{jk}\nabla(t_{jk})\times\partial_{jk}^{2}u)+\sum_{m=0}^{3}B_{3,m}(u,v,w).\nonumber 
\end{align}
Now, if we interchange $P_{1}$, $P_{2}$ on (\ref{eq:2.6}), and
use 
\[
P_{2}P_{1}=P_{1}P_{2}-[P_{1,}P_{2}],
\]
where $[P_{1},P_{2}]$ is the commutator of two second order elliptic
operators, then $[P_{1},P_{2}]$ is a third order differential operator.
Thus, (\ref{eq:2.6}) becomes
\begin{eqnarray}
P_{2}P_{1}(x,D)v & = & -P_{1}(x,D)[\sum_{jk}\nabla(t_{jk})\cdot\partial_{jk}^{2}u]\label{eq:2.8}\\
 &  & +\sum_{m=0}^{3}\widetilde{B_{2,m}}(u,v,w),\nonumber 
\end{eqnarray}
where $\widetilde{B_{2,m}}$ is an $m$-th order differential operator
and 
\[
\sum_{m=0}^{3}\widetilde{B_{2,m}}(u,v,w)=\sum_{m=0}^{3}B_{2,m}(u,v,w)-[P_{1},P_{2}](x,D)v.
\]
Now, combine (\ref{eq:2.5}), (\ref{eq:2.7}) and (\ref{eq:2.8})
together, we have 
\begin{eqnarray}
P_{2}P_{1}(x,D)\left(\begin{array}{c}
u\\
v\\
w
\end{array}\right) & = & -\left(\begin{array}{c}
0\\
P_{1}(x,D)[\sum_{jk}\nabla(t_{jk})\cdot\partial_{jk}^{2}u]\\
P_{2}(x,D)[\sum_{jk}\nabla(t_{jk})\times\partial_{jk}^{2}u]
\end{array}\right)\label{eq:2.9}\\
 &  & +\sum_{m=0}^{3}\left(\begin{array}{c}
B_{1,m}(u,v)\\
\widetilde{B_{2,m}}(u,v,w)\\
B_{3,m}(u,v,w)
\end{array}\right).\nonumber 
\end{eqnarray}

Now, for $P_{1}\left(\sum_{jk}\nabla(t_{jk})\cdot\partial_{jk}^{2}u\right)$
in (\ref{eq:2.9}), recall that $P_{1}(x,D)=R+\mu\Delta$ and $Ru=\nabla\cdot(\nabla uT)$,
then we have 
\begin{equation}
P_{1}(x,D)[\sum_{jk}\nabla(t_{jk})\cdot\partial_{jk}^{2}u]=R(\sum_{jk}\nabla(t_{jk})\cdot\partial_{jk}^{2}u)+\mu\Delta(\sum_{jk}\nabla(t_{jk})\cdot\partial_{jk}^{2}u).\label{eq:2.10}
\end{equation}
 For the second term of (\ref{eq:2.10}), by using the vector identity
$\Delta u=\nabla(\nabla\cdot u)-\nabla\times\nabla\times u=\nabla v-\nabla\times w$,
it is easy to see 
\begin{align*}
 & \Delta(\sum_{jk}\nabla(t_{jk})\cdot\partial_{jk}^{2}u)\\
 & =\sum_{jk}\nabla(t_{jk})\cdot\partial_{jk}^{2}(\Delta u)+\widetilde{A_{2,3}}(u)+\widetilde{A_{2,2}}(u)\\
 & =\sum_{jk}\nabla(t_{jk})\cdot\partial_{jk}^{2}(\nabla v-\nabla\times w)+\widetilde{A_{2,3}}(u)+\widetilde{A_{2,2}}(u)\\
 & =\widetilde{B_{2,3}}(u,v,w)+\widetilde{A_{2,2}}(u),
\end{align*}
where $\widetilde{A_{2,m}}$ and $\widetilde{B_{2,m}}$ are $m$-th
order differential operators and 
\[
\widetilde{B_{2,3}}(u,v,w)=\sum_{jk}\nabla(t_{jk})\cdot\partial_{jk}^{2}(\nabla v-\nabla\times w)+\widetilde{A_{2,3}}(u).
\]
For the first term of (\ref{eq:2.10}), we have
\begin{align*}
 & R(\sum_{jk}\nabla(t_{jk})\cdot\partial_{jk}^{2}u)\\
 & =\sum_{\ell m}t_{\ell m}\partial_{\ell m}^{2}(\sum_{jk}\nabla(t_{jk})\cdot\partial_{jk}^{2}u))\\
 & =\sum_{jk}\nabla(t_{jk})\cdot[\sum_{\ell m}t_{\ell m}\partial_{\ell m}^{2}\partial_{jk}^{2}u]+\widetilde{C_{2,3}}(u)+\widetilde{C_{2,2}}(u)\\
 & =\sum_{jk}\nabla(t_{jk})\cdot\partial_{jk}^{2}(\sum_{\ell m}t_{\ell m}\partial_{\ell m}^{2}u)+\widetilde{D_{2,3}}(u)+\widetilde{D_{2,2}}(u)\\
 & =\sum_{jk}\nabla(t_{jk})\cdot Ru+\widetilde{D_{2,3}}(u)+\widetilde{D_{2,2}}(u),
\end{align*}
and use (\ref{eq:2.2.}), we have $Ru=-\mu\Delta u+A_{1,1,}(u,v)+A_{1,0}(u,v)$,
we have 

\begin{align*}
 & R(\sum_{jk}\nabla(t_{jk})\cdot\partial_{jk}^{2}u)\\
 & =\sum_{jk}\nabla(t_{jk})\cdot\partial_{jk}^{2}(-\mu\Delta u+A_{1,1}(u,v)+A_{1,0}(u,v))+\widetilde{D_{2,3}}(u)+\widetilde{D_{2,2}}(u)\\
 & =\sum_{jk}\nabla(t_{jk})\cdot\partial_{jk}^{2}(-\mu(\nabla v-\nabla\times w))+\sum_{m=0}^{3}\widetilde{E_{2,3}}(u,v)\\
 & =\sum_{m=0}^{3}\widetilde{F_{2,m}}(u,v,w)
\end{align*}
where $\widetilde{C_{2,m}},\widetilde{D_{2,m}},\widetilde{E_{2,m}}$
and $\widetilde{F_{2,m}}$ are $m$-th order differential operators.
From the above calculation and (\ref{eq:2.9}), we have 
\begin{equation}
P_{2}P_{1}(x,D)v=\sum_{m=0}^{3}\widehat{E_{2,m}}(u,v,w),\label{eq:2.11}
\end{equation}
where $\widehat{E_{2,m}}$ are $m$-th order differential operators.
Similarly, for $P_{2}\left(\sum_{jk}\nabla(t_{jk})\times\partial_{jk}^{2}u\right)$,
it is easy too see that 
\[
\Delta(\sum_{jk}\nabla(t_{jk})\times\partial_{jk}^{2}u)=\widetilde{A_{3,3}}(u,v,w)+\widetilde{A_{3,2}}(u),
\]
where $\widetilde{A_{3,m}}$ is an $m$-th order differential operator.
Similarly, for $R(\sum_{jk}\nabla(t_{jk})\times\partial_{jk}^{2}u)$,
component-wise, we have
\begin{align*}
 & \left[R(\sum_{jk}\nabla(t_{jk})\times\partial_{jk}^{2}u)\right]_{i}\\
 & =\sum_{\ell m}t_{\ell m}\partial_{\ell m}^{2}(\sum_{jk}\nabla(t_{jk})\times\partial_{jk}^{2}u)_{i}\\
 & =\sum_{\ell m}\sum_{jk}(\nabla(t_{jk})\times t_{\ell m}\partial_{\ell m}^{2}\partial_{jk}^{2}u)_{i}+\widetilde{B_{3,3}}(u)+\widetilde{B_{3,2}}(u)\\
 & =\sum_{jk}(\nabla(t_{jk})\times\partial_{jk}^{2}(\sum_{\ell m}t_{\ell m}\partial_{\ell m}^{2}u)_{i})+\widetilde{C_{3,3}}(u)+\widetilde{C_{3,2}}(u)\\
 & =\sum_{jk}(\nabla(t_{jk})\times\partial_{jk}^{2}Ru)_{i}+\widetilde{D_{3,3}}(u)+\widetilde{D_{3,2}}(u)
\end{align*}
and use (\ref{eq:2.2.}) again, we obtain

\begin{align*}
 & \left[R(\sum_{jk}\nabla(t_{jk})\times\partial_{jk}^{2}u)\right]_{i}\\
 & =\sum_{jk}\left(\nabla(t_{jk})\times\partial_{jk}^{2}[-\mu(\nabla v-\nabla\times w)]\right)_{i}+\sum_{m=0}^{3}\widetilde{E_{3,m}}(u,v)\\
 & =\sum_{m=0}^{3}\widetilde{F_{3,m}}(u,v,w),
\end{align*}
where $\widetilde{B_{3,m}},\widetilde{C_{3,m}},\widetilde{D_{3,m}},\widetilde{E_{3,m}}$
and $\widetilde{F_{3,m}}$ are $m$-th order differential operators.

Therefore, we transform the equation (\ref{eq:2.7}) into 
\begin{equation}
P_{2}P_{1}(x,D)w=\sum_{m=0}^{3}\widehat{E_{3,m}}(u,v,w),\label{eq:2.12}
\end{equation}
where $\widehat{E_{3,m}}$ are $m$-th order differential operators.
From (\ref{eq:2.11}), (\ref{eq:2.12}) and (\ref{eq:2.9}), we can
obtain 
\[
P_{2}P_{1}(x,D)\left(\begin{array}{c}
u\\
v\\
w
\end{array}\right)=\sum_{m=0}^{3}\left(\begin{array}{c}
\widehat{E_{1,m}}(u,v,w)\\
\widehat{E_{2,m}}(u,v,w)\\
\widehat{E_{3,m}}(u,v,w)
\end{array}\right),
\]
with $\widehat{E_{\ell,m}}$ are $m$-th order differential operators,
or equivalently, 
\begin{equation}
P_{2}P_{1}U=\sum_{m=0}^{3}\widehat{E_{m}}(U),\label{eq:2.13}
\end{equation}
with $\widehat{E_{m}}=(\widehat{E_{3,m}},\widehat{E_{3,m}},\widehat{E_{3,m}})^{t}$
is an $m$-th order differential operator and $U=(u,v,w)^{t}$, which
means this fourth-order differential equation has the same leading
term $P_{2}P_{1}$ and all coefficients of (\ref{eq:2.13}) lie in
$G^{s}$. Moreover, use the elliptic regularity for (\ref{eq:2.13})
with Gevrey coefficients, then $U\in G^{s}$ by Proposition 2.13 in
\cite{boutet1967pseudo}.

\section{The asymptotic behavior of $u$ near 0}

As in Section 2, we set $U=(u,v,w)^{t}$, where $v=\nabla\cdot u$
and $w=\nabla\times u$. If we can prove that $U$ solves (\ref{eq:2.13})
and satisfies the SUCP, then $u$ solves (\ref{eq:1.1}) and fulfills
the SUCP. In the following lemma, we describe the asymptotic behavior
of $u$ near $0$. Recall that if $u\in H_{loc}^{2}(\Omega;\mathbb{R}^{3})$,
then $u\in C^{\infty}(\Omega)$ by the standard elliptic regularity.
Thus, $\forall k\in\mathbb{N}$, we can consider $u\in H_{loc}^{k}(\Omega;\mathbb{R}^{3}$)
for arbitrary $k\in\mathbb{N}$ in the following results.
\begin{lem}
\cite{lin2004strong} Let $u$ be a solution to (\ref{eq:1.1}) and
for all $N>0$
\[
\int_{R\leq|x|\leq2R}|u|^{2}dx=O(R^{N})\mbox{ as }R\to0.
\]
Then for $|\beta|\leq2$, we have 
\[
\int_{R\leq|x|\leq2R}|R^{|\beta|}D^{\beta}u|^{2}dx=O(R^{N})\mbox{ as }R\to0.
\]
\end{lem}
\begin{proof}
The lemma was proved by the Corollary 17.1.4 in H$\ddot{\mathrm{o}}$rmander
\cite{hormander2007analysis}. 

By using the lemma 3.1, we will get the following Corollary.\end{proof}
\begin{cor}
Let $U=(u,v,w)^{t}$ with $v=\nabla\cdot u$ and $w=\nabla\times u$.
Then for $|\beta|\leq1$, $\forall N>0$, we have 
\begin{equation}
\int_{R\leq|x|\leq2R}|D^{\beta}U|^{2}dx=O(R^{N})\mbox{ as }R\to0.\label{eq:3.1}
\end{equation}

\end{cor}
In fact, we can get higher derivatives for $|\beta|\geq2$ in the
Corollary 3.2.
\begin{lem}
\cite{hormander2007analysis} If $U$ satisfies a fourth order strongly
elliptic system (\ref{eq:2.13}) 
\[
PU=\sum_{m=0}^{3}\widehat{E_{m}}(U),
\]
and $U$ satisfies $\forall N>0$, 
\[
\int_{R\leq|x|\leq2R}|U|^{2}dx=O(R^{N})\mbox{ as }R\to0.
\]
Then it follows that if $|\beta|\leq4$ that 
\begin{equation}
\int_{R\leq|x|\leq2R}|R^{|\beta|}D^{\beta}U|^{2}dx=O(R^{N})\mbox{ as }R\to0.\label{eq:3.2}
\end{equation}
\end{lem}
\begin{proof}
Since $U$ satisfies (\ref{eq:2.13}), a fourth order strongly elliptic
system, by using the Corollary 17.1.4 in \cite{hormander2007analysis},
we can obtain (\ref{eq:3.2}).\end{proof}
\begin{rem}
In the section 3 of \cite{lin2004strong}, the author proved (\ref{eq:3.2})
holding for $|\beta|\leq2$. From Lemma 3.3 and the coefficients of
$P$ are in the Gevrey class $G^{s}$, we have $U\in G^{s}$ and 
\[
\int_{|x|\leq R}|D^{\beta}U|^{2}dx=O(R^{N})\mbox{ as }R\to0,
\]
for $|\beta|\leq4$ and $\forall N>0$. 
\end{rem}

\section{Proof of the main theorem}

In this section, we want to prove Theorem 1.1. If $U=(u,v,w)^{t}$
satisfies (\ref{eq:2.13}) and the SUCP, then the SUCP holds for $u$,
where $u$ fulfills (\ref{eq:1.1}).

\subsection{SUCP for $U$}

In the following theorem, we will prove the SUCP for $U$.
\begin{thm}
Suppose that the second order elliptic operators $P_{\ell}$ satisfies
(\ref{eq:1.3}), (\ref{eq:1.4}), (\ref{eq:1.5}) and (\ref{eq:1.6})
for $\ell=1,2$. $\alpha>0$ satisfies (\ref{eq:1.7}) at $x=0$ and
$s$ satisfies (\ref{eq:1.8}). Let $P=P_{2}P_{1}$ be a fourth order
elliptic operator. Then the SUCP holds for the elliptic system 
\begin{equation}
PU=\sum_{|\beta|\leq3}a_{\beta}\partial^{\beta}U\label{eq:4.1}
\end{equation}
provided the coefficients of $P_{\ell}$ are in the Gevery class $G^{s}$.\end{thm}
\begin{proof}
The proof follows from \cite{colombini2010strong} and section 1.
To prove Theorem 4.1, there are two steps. First, Gevrey regularity
of the elliptic system implies the solution $U$ of (\ref{eq:4.1})
is in the Gevrey class $G^{s}$ (see Proposition 2.13 in \cite{boutet1967pseudo}).
Use the vanishing order assumption and $U\in G^{s}$, we have 
\begin{equation}
|U|\lesssim e^{-|x|^{-\gamma}},\label{eq:4.2}
\end{equation}
near $x=0$ and for some constant $\gamma>0$ (see Appendix). Second,
we can show that (\ref{eq:4.2}) implies $U$ vanishes near $0$ by
using appropriate Carleman estimates. In addition, since $U$ vanishes
near $0$, by the results in \cite{uhlmann2009unique}, we have $U\equiv0$
in $\Omega$.
\end{proof}

\subsection{Carleman Estimates}

We are going to derive the Carleman estimates for the weight $e^{\tau|x|^{-\alpha}}$
for the fourth order elliptic operator $P=P_{2}P_{1}$ in this section.
The following Carleman estimates for the scalar case has been proven
in \cite{colombini2006strong} and \cite{colombini2010strong}. Similar
to the scalar elliptic equation, we can derive the following Carleman
estimate for the special elliptic system.
\begin{prop}
Let $P_{\ell}(x,D)=\sum_{jk}a_{jk}^{\ell}(x)\partial_{jk}^{2}$ be
a principally diagonal second order elliptic operator where $a_{jk}^{\ell}(x)\in G^{s}$
satisfies(\ref{eq:1.3}), (\ref{eq:1.4}), (\ref{eq:1.5}) and (\ref{eq:1.6})
for $\ell=1,2$. $\alpha>0$ satisfies (\ref{eq:1.7}) at $x=0$ and
$s$ satisfies (\ref{eq:1.8}). Then there exist $\tau_{0}>0$ and
$r_{0}>0$ such that for $\tau>\tau_{0}$ and for all $V\in C^{\infty}((B_{r_{0}}\backslash\{0\});\mathbb{R}^{7})$,
$\ell=1,2$, the following inequality holds:
\begin{align*}
 & \tau\int|D^{2}(|x|^{\alpha/2}e^{\tau|x|^{-\alpha}}V|^{2}dx+\tau^{3}\int|x|^{-4-3\alpha}e^{2\tau|x|^{-\alpha}}V|^{2}dx\\
 & \lesssim\int|e^{2\tau|x|^{-\alpha}}(P_{\ell}V)|^{2}dx.
\end{align*}
\end{prop}
\begin{proof}
Since $P_{\ell}$ is the principally diagonal second order elliptic
operator for $\ell=1,2$, we can directly follow the consequences
in \cite{colombini2010strong} and use the proof in \cite{colombini2006strong}.
For more details and classical results, we refer readers to \cite{hormander1963linear,tataru1999carleman}.
\end{proof}
By using the integration by parts, we can get a stronger inequality
in the following. For more details, we refer readers to \cite{colombini2010strong}
and section 3, then we have
\begin{align*}
 & \sum_{j=0}^{2}\tau^{3-2j}\int e^{2\tau|x|^{-\alpha}}|x|^{\alpha}|x|^{2(j-2)(1+\alpha)}|D^{j}V|^{2}dx\\
 & \lesssim\int|e^{2\tau|x|^{-\alpha}}|P_{\ell}V|^{2}dx,
\end{align*}
with $P_{\ell}$ satisfying all the assumptions in Proposition 4.2
for $\ell=1,2$. Note that the right hand side of (\ref{eq:2.3})
and (\ref{eq:2.4}) involve second order derivatives of $u$, we cannot
apply the Carleman estimates for the second order differential systems
directly to get the SUCP for $U$. Since we have transformed (\ref{eq:1.1})
into a special fourth order elliptic system with the same leading
operator, see (\ref{eq:2.13}), then we can derive the Carleman estimates
for the operator $P=P_{2}P_{1}$.
\begin{cor}
\cite{colombini2010strong} Let 
\[
A=\sum_{jk}a_{jk}(x)\partial_{x_{j}x_{k}}^{2}
\]
be a second order strongly elliptic operator with $a_{ij}$ in the
Gevrey class $G^{s}$. Suppose $\alpha>0$ satisfying (\ref{eq:1.7})
at $x=0$. Then there exists $\tau_{0}$ such that for all $|s|$,
$k\leq\nu$ and $\tau\geq\tau_{0}$
\begin{align}
 & \sum_{j=0}^{k+2}\tau^{3-2j}\int|x|^{\alpha+2j(1+\alpha)}|x|^{2s}e^{2\tau|x|^{-\alpha}}|D^{j}V|^{2}dx\label{eq:4.3}\\
 & \lesssim\sum_{j=0}^{k}\tau^{-2j}\int|x|^{2(2+j)(1+\alpha)}|x|^{2s}e^{2\tau|x|^{-\alpha}}|D^{j}(AV)|^{2}dx.\nonumber 
\end{align}
 \end{cor}
\begin{proof}
See \cite{colombini2010strong} and section 3. We can use the induction
hypothesis to prove the Corollary 4.3.
\end{proof}
For the fourth order elliptic operator $PU=P_{2}P_{1}U$ is the product
of two second order elliptic operators which satisfies (\ref{eq:2.13}),
where $U=(u,v,w)^{t}$ and $P_{\ell}(x,D)U=\sum_{jk}a_{jk}^{\ell}(x)\partial_{jk}^{2}U$.
Recall that $a_{jk}^{\ell}\in G^{s}$ and $\alpha>0$ satisfying (\ref{eq:1.7})
uniformly in $x$ and for $\ell=1,2$. Apply the Corollary 4.2 iteratively,
then we have
\begin{align}
 & \sum_{j=0}^{4}\tau^{6-2j}\int|x|^{-8-6\alpha}|x|^{2j(1+\alpha)}e^{2\tau|x|^{-\alpha}}|D^{j}V|^{2}dx\label{eq:4.4}\\
 & \lesssim\sum_{j=9}^{2}\tau^{3-2j}\int|x|^{-4-3\alpha}|x|^{j(1+\alpha)}e^{2\tau|x|^{-\alpha}}|D^{j}(P_{1}V)|^{2}dx\nonumber \\
 & \lesssim\int e^{2\tau|x|^{-\alpha}}|(P_{2}P_{1}V)|^{2}dx=\int e^{2\tau|x|^{-\alpha}}|PV|^{2}dx,\nonumber 
\end{align}
where the first inequality is obtained by (\ref{eq:4.3}) with $k=2$,
$s=-4-\dfrac{7}{2}\alpha$ and the second inequality is obtained by
(\ref{eq:4.3}) with $k=0$, $s=-2(1+\alpha)$.

Now, we want to prove the SUCP for (\ref{eq:1.1}). Here we prove
the theorem 2.2.\\
\textit{Proof of Theorem 4.1}: The operator $P=P_{2}P_{1}$ is strongly
elliptic in the Gevrey class $G^{s}$, then $U$ is also in the Gevrey
class $G^{s}$. Therefore, we have the vanishing of infinite order
implies that 
\[
|u|\lesssim e^{-|x|^{-\gamma}}
\]
for some $\gamma>\alpha$. Let $\chi\in C_{0}^{\infty}(\mathbb{R}^{3})$
be such that $\chi\equiv1$ for $|x|\leq R$ and $\chi\equiv0$ for
$|x|\geq2R$ ($R>0$ is small enough. Then we can apply (\ref{eq:4.4})
to the function $\chi U$, which means 
\begin{align}
 & C\sum_{|\beta|=0}^{4}\tau^{6-2|\beta|}\int_{|x|<R}|x|^{(2|\beta|-6)(1+\alpha)-2}e^{2\tau|x|^{-\alpha}}|D^{\beta}U|^{2}dx\label{eq:4.5}\\
 & \leq\int e^{2\tau|x|^{-\alpha}}|PU|^{2}dx\nonumber \\
 & \leq\int_{|x|<R}e^{2\tau|x|^{-\alpha}}|PU|^{2}dx+\int_{|x|>R}e^{2\tau|x|^{-\alpha}}|P(\chi U)|^{2}\nonumber \\
 & \leq\int_{|x|<R}e^{2\tau|x|^{-\alpha}}|\sum_{m=0}^{3}\widehat{E_{m}}(U)|^{2}dx+\int_{|x|>R}e^{2\tau|x|^{-\alpha}}|P(\chi U)|^{2},\nonumber 
\end{align}
by using the reduction elliptic system (\ref{eq:2.13}). 

If $\tau$ is large and $R$ is sufficiently small, then (\ref{eq:4.5})
implies 
\begin{align}
 & C\sum_{|\beta|=0}^{4}\tau^{6-2|\beta|}\int_{|x|<R}|x|^{(2|\beta|-6)(1+\alpha)-2}e^{2\tau|x|^{-\alpha}}|D^{\beta}U|^{2}dx\label{eq:4.6}\\
 & \leq\int_{|x|>R}e^{2\tau|x|^{-\alpha}}|P(\chi U)|^{2},\nonumber 
\end{align}
for some constant $C>0$. Notice that $e^{\tau|x|^{-\alpha}}\geq e^{\tau R^{-\alpha}}$
for $|x|<R$ and $e^{\tau|x|^{-\alpha}}\leq e^{\tau R^{-\alpha}}$
for $|x|>R$. Therefore, we can use (\ref{eq:4.6}) to obtain
\begin{align*}
 & C\sum_{|\beta|=0}^{4}\tau^{6-2|\beta|}\int_{|x|<R}|x|^{(2|\beta|-6)(1+\alpha)-2}|D^{\beta}U|^{2}dx\\
 & \leq\int_{|x|>R}|P(\chi U)|^{2}.
\end{align*}
Let $\tau\to\infty$, we get $U=0$ in $\{|x|<R\}$ for $R$ small,
which implies $u=0$ in $\{|x|<R\}$. Furthermore, by using the unique
continuation principal in \cite{uhlmann2009unique}, we can obtain
$u\equiv0$ in $\Omega$, then we are done.

\section{Appendix}

In this section, we state some properties of Gevrey functions. For
more details, see \cite{lerner1981resultats,boutet1967pseudo}.
\begin{lem}
Let $U$ be a bounded open set and suppose that $0\in U$, $s\geq1$
and $f\in G^{s}(U)$ satisfies 
\[
\partial^{\beta}f(0)=0
\]
for all multiindices $\beta$. Let $s-1<\rho$, then 
\[
|f(x)|\leq e^{-|x|^{-1/\rho}}
\]
near $x=0$.
\end{lem}

\begin{lem}
We have 
\[
e^{-|x|^{-1/\rho}}\in G^{s}(\mathbb{R}^{3})
\]
provided $1+\rho=s$.
\end{lem}

\begin{lem}
Let 
\[
P(x,D)u=f\mbox{ in }U
\]
 be an elliptic differential system with coefficients and right had
side in the Gevrey class $G^{s}(U)$. Then $u\in G^{s}(V)$ for all
bounded $V\Subset U$.\end{lem}
\begin{proof}
See \cite{boutet1967pseudo}, Proposition 2.13, we know that the Gevrey
class are good classes of elliptic regularity.
\end{proof}
\bibliographystyle{plain}
\bibliography{ref}

\end{document}